\documentclass[letter,reqno,11pt]{amsart}
\usepackage[margin=1in]{geometry}
\usepackage[utf8]{inputenc}
\usepackage{graphicx, amsfonts, amssymb, amsthm, mathtools}
\usepackage{soul}
\usepackage{hyperref}
\usepackage{enumerate}
\usepackage{tabularx}
\usepackage{verbatim}
\usepackage{colortbl}
\usepackage[nameinlink]{cleveref}
\hypersetup{	colorlinks,	linkcolor={blue},	citecolor={blue},urlcolor={blue}}

\newtheorem{theorem}{Theorem}[section]

\newtheorem{definition}[theorem]{Definition}
\newtheorem{lemma}[theorem]{Lemma}

\newtheorem{corollary}[theorem]{Corollary}

\newtheorem{remark}[theorem]{Remark}
\numberwithin{equation}{section}

\newcommand{\R}{\mathbb{R}}

\newcommand{\N}{\mathbb{N}}
\newcommand{\UH}{\mathbb{H}}
\newcommand{\mon}{\operatorname{Mon}(\pi)}
\newcommand{\pslz}{\operatorname{PSL}_2(\mathbb{Z})}
\newcommand{\pslr}{\operatorname{PSL}_2(\mathbb{R})}
\newcommand{\kernel}{\operatorname{ker}(\varphi)}

\newlength\tindent
\title{Explicit Constructions of Finite Groups as Monodromy Groups}
\author{Ra-Zakee Muhammad, Javier Santiago and Eyob Tsegaye}

\address{Department of Mathematics, Pomona College, 333 N. College Way
Claremont, CA 91711}

\email{rmab2018@mymail.pomona.edu}

\address{Department of Computer Science,
University of Puerto Rico, Río Piedras,
17 Ave Universidad, Ste 1701, 
San Juan, Puerto Rico 00925-2537}

\email{javier.santiago16@upr.edu}

\address{Deparment of Mathematics, Stanford University, Building 380, Stanford, California 94305}

\email{etsegay@stanford.edu}

\begin{document}

\maketitle

\begin{abstract} 
In 1963, Greenberg proved that every finite group appears as the monodromy group of some morphism of Riemann surfaces. In this paper, we give two constructive proofs of Greenberg's result. First, we utilize free groups, which given with the universal property and their construction as discrete subgroups of $\pslr$, yield a very natural realization of finite groups as monodromy groups. We also give a proof of Greenberg's result based on triangle groups $\Delta(m, n, k)$. Given any finite group $G$, we make use of subgroups of $\Delta(m, n, k)$ in order to explicitly find a morphism $\pi$ such that $G \simeq \mon$.
\end{abstract}

\section{Introduction}

Greenberg's celebrated result \cite{greenberg1963maximal} from 1963 shows than any finite group can be realized as the monodromy group of some morphism of Riemann surfaces.  Over $40$ years later, in \cite{wolfart2005abc}, Wolfart provided a method to make Greenberg's ideas explicit for finite cyclic groups. Specifically, he utilizes surjections from triangle groups in order to construct the  cyclic group as a monodromy group. Inspired by Wolfart, in this paper we seek to make Greenberg's result explicit for all finite groups. That is, given any finite group $G$, our goal is to find an explicit morphism $\pi$ with monodromy group $\mon$ such that $G \simeq \mon$. On that account, we will be presenting two constructive proofs of Greenberg's classical result. 

We first present a natural construction arising from the theory of free groups, where we make use of the appearance of all groups $G$ as quotient groups of free groups  \cite{dummit2004abstract}, as well as the identification of the free group of rank $2$ with the congruence subgroup $\Gamma(2)$ \cite{girondo2012introduction}.
Then we give a proof of our main result, which is a construction which utilizes surjections from appropriate subgroups of triangle groups. The formal construction is as follows:

\begin{theorem} \label{mainResultTriangle}
Let $S_n$ be the symmetric group of degree $n$ and let $\varphi \colon \Delta(2, n, n-1) \to S_n$ be the homomorphism given by
\begin{equation*}
    \gamma_0 \mapsto (1 \; 2), \quad \gamma_1 \mapsto (1 \; 2 \; \cdots \; n), \quad \gamma_{\infty} \mapsto (n \; \cdots \; 3 \; 2).
\end{equation*}

Then, 

\begin{enumerate}[\normalfont(i)]
    \item $\varphi$ is a surjective group homomorphism. Moreover, if $n < 5$, then $\varphi$ is an isomorphism.
    %$\Delta(2, n, n-1) \simeq S_n$.
    
    \item If $G$ is a subgroup of $S_n$, then  $\Gamma = \varphi^{-1}(G)$ and $\Gamma^\prime = \kernel$ are both Fuchsian groups.
    
    \item $G \simeq \mon$, where $\pi \colon \Gamma^\prime \backslash \UH \to \Gamma \backslash \UH$ is the canonical projection and $G$, $\Gamma$, and $\Gamma'$ are as stated in (ii).
\end{enumerate}

\end{theorem}

Note that by Cayley's theorem, we may identify any finite group $G$ as a subgroup of $S_n$ for appropriate $n$. The organization of the paper is the following. In \Cref{backgroundSection}, we define the concepts of Fuchsian groups and triangle groups, in addition to giving a description of monodromy in simple algebraic terms. Moreover, we also go into detail on the explicit constructions of free groups $F_r$ as subgroups of $\Gamma(2)$. In \Cref{constructions}, we then give both constructive proofs of Greenberg's result. 

\section{Preliminaries} \label{backgroundSection}

Recall $\operatorname{SL}_2(\R)$, as a topological space, can be naturally identified with a subset of $\R^4$. Consequently, $\pslr$ is endowed with a quotient topology. In this paper, we will be interested in subgroups $\Gamma \leq \pslr$ such that the topology induced in $\Gamma$ is discrete. 

\begin{definition} \cite{girondo2012introduction}
A \textbf{Fuchsian group} is a subgroup $\Gamma \leq \pslr$ such that for all $\gamma \in \Gamma$ there is a neighbourhood $V$ of $\gamma$ in $\operatorname{PSL}_2(\mathbb{R})$ with $V \cap \Gamma = \{ \gamma \}$. 
\end{definition}

By definition, it is clear that if $\Gamma$ is a Fuchsian group and $\Gamma^\prime \leq \Gamma$, then $\Gamma^\prime$ is also a Fuchsian group. This will be relevant when working with subgroups of $\pslz$, since $\pslz$ is a discrete subgroup of $\pslr$. For more details regarding Fuchsian groups, see \cite{katok1992fuchsian, jones1987complex, garcia2001topics}.

As mentioned before, the goal of this paper is to explicitly realize any finite group $G$ as the monodromy group of some branched cover. Now, on one hand, given a branched cover of Riemann surfaces, we may realize the monodromy group by lifting paths around the branch points. However, we may also take an algebraic point of view. In \cite{girondo2012introduction}, we see that monodromy groups can also be explicitly constructed from Fuchsian groups. Namely, given two Fuchsian groups $\Gamma^\prime, \Gamma$ with $\Gamma^\prime \leq \Gamma$, we have $\displaystyle{\frac{\Gamma}{\bigcap_{\gamma \in \Gamma} \gamma^{-1} \Gamma^\prime \gamma} \simeq \mon}$, where $\pi \colon \Gamma^\prime \backslash \UH \to \Gamma \backslash \UH$. Furthermore, $\displaystyle{\bigcap_{\gamma \in \Gamma} \gamma^{-1} \Gamma^\prime \gamma}$ is the largest normal subgroup of $\Gamma$ that is contained in $\Gamma^\prime$ (also called the normal core). Therefore, if $\Gamma^\prime$ is a normal subgroup of $\Gamma$, we have the following simple description of $\mon$.

\begin{corollary} \label{normalIntersection} \cite[Corollary 2.59]{girondo2012introduction} Let $\Gamma^\prime, \Gamma$ be Fuchsian groups. If $\Gamma^\prime \trianglelefteq \Gamma$, then $\Gamma / \Gamma^\prime \simeq \mon$, where $\pi \colon \Gamma^\prime \backslash \UH \to \Gamma \backslash \UH$.
\end{corollary}

Making Greenberg's result explicit is then a purely algebraic question: Given any finite group $G$, can we find Fuchsian groups $\Gamma^\prime, \Gamma$ with $\Gamma^\prime \trianglelefteq \Gamma$ such that $G \simeq \Gamma / \Gamma^\prime$? \Cref{mainResultFree} and \Cref{mainResultTriangle} provide two separate constructions that answer this question.

In \Cref{mainResultFree}, we take $\Gamma$ to be a free group, whose rank depends on the number of generators of the finite group $G$. This construction is based on both the fact that $\Gamma(2)$ is a free group generated by $A = \pm \begin{bmatrix} 1 & 2 \\ 0 & 1 \end{bmatrix}, B = \pm \begin{bmatrix} 1 & 0 \\ 2 & 1 \end{bmatrix} \in \pslz$ \cite{birch1994theory}, as well as the classical result that the free group of rank two contains freely generated subgroups of rank $r$ for any $r \in \N$. The following lemma gives an explicit construction of free groups of rank $r$ as subgroups of $\Gamma(2)$.

\begin{lemma} \label{frContainsFtwo} \cite[Proposition 7.80]{drutu2018geometric} Let $A = \pm \begin{bmatrix} 1 & 2 \\ 0 & 1 \end{bmatrix}$ and $B = \pm \begin{bmatrix} 1 & 0 \\ 2 & 1 \end{bmatrix}$ and let $F_r = \langle \; B^{-j}AB^j \; \vert \; j = 0, 1, \dots, r-1 \; \rangle$. Then $F_r$ is a free subgroup of rank $r$ in  $\Gamma(2)$.
\end{lemma}

\begin{comment}
\begin{proof}
Let $X_j = B^{-j}AB^j$, for $0 \leq j \leq r-1$. We may observe that any reduced word in terms of the $X_j$ is non-trivial. Hence, $F_r$ is freely generated by $X_0, X_1, \dots, X_{r-1}$. For more details on this proof, see Proposition 7.80 in \cite{drutu2018geometric}. 
\end{proof}
\end{comment}

For \Cref{mainResultTriangle} we choose $\Gamma$ to be a partiular subgroup of a triangle group.

\begin{definition}
Let $m, n, k \in \N$. Then a group with presentation

\begin{equation*}
    \Delta(m, n, k) = \langle \gamma_0, \gamma_1, \gamma_{\infty} \; \vert \; \gamma_0^m = \gamma_1^n = \gamma_{\infty}^k = \gamma_0 \gamma_1 \gamma_{\infty} = 1 \rangle
\end{equation*}

is called a \textbf{triangle group}.
\end{definition}

An important result regarding triangle groups is the following: 

\begin{theorem}\cite{beardon2012geometry} \label{triangleFuchsian}
Let $m, n, k \in \N$. Then the triangle group $\Delta(m, n, k)$ is a Fuchsian group. 
\end{theorem}

\begin{remark}\cite{girondo2012introduction}
Notice $\Gamma(2)$ can be identified with the triangle group $\Delta(\infty, \infty, \infty)$.
\end{remark}

In both \Cref{mainResultFree} and \Cref{mainResultTriangle}, we will construct a surjective homomorphism $\varphi$ from $\Gamma$ onto the finite group $G$, and $\kernel$ will be the natural choice for the normal subgroup $\Gamma^\prime$.

\section{Constructions} \label{constructions}

In this section, we will be using \Cref{normalIntersection} to explicitly realize finite groups $G$ as monodromy groups. As discussed in \Cref{backgroundSection}, given a finite group $G$, our approach is to find a Fuchsian group $\Gamma$ and a surjective homomorphism $\varphi \colon \Gamma \to G$ such that $G \simeq \Gamma / \Gamma^\prime$, where $\Gamma^\prime = \kernel$. The kernel of such a homomorphism $\varphi$ is then a normal Fuchsian subgroup of $\Gamma$. As a result, we can then use \Cref{normalIntersection} to conclude $G \simeq \mon$, where $\pi \colon \Gamma^\prime \backslash \UH \to  \Gamma \backslash \UH$ is the canonical map. Two separate constructions will be provided, one with free groups and one with subgroups of triangle groups.

\subsection{Free group construction} \label{freeSection}

From group theory, we have that any group $G$ is isomorphic to a quotient group of some free groups. Here, we will be using the structure of groups as quotients of free groups in order to explicitly express finite groups as monodromy groups.

Consider both $\Gamma(2) = \langle A, B \rangle$ and $F_r$ as constructed in \Cref{frContainsFtwo}.

\begin{theorem} \label{mainResultFree}
Let $G$ be an $r$-generated finite group, $\Gamma = F_r$ and $\varphi$ a surjective group homomorphism $\varphi \colon \Gamma \to G$ with $\Gamma' = \ker\varphi$. Then, 

\begin{enumerate}[\normalfont(i)]

    \item $\Gamma$ and $\Gamma^\prime$ are both Fuchsian groups. Moreover, $\Gamma^\prime$ is a free group of rank $1 + \lvert G \rvert(r-1)$.
    
    \item $G \simeq \mon$, where $\pi \colon \Gamma^\prime \backslash \UH \to \Gamma \backslash \UH$ is the canonical projection.
\end{enumerate}

\end{theorem}

\begin{proof} 
$(i)$ By the construction in \Cref{frContainsFtwo}, we see that $\Gamma \leq \pslz$, and thus, it is a Fuchsian group. Since $\Gamma^\prime = \kernel$ is a subgroup of the Fuchsian group $\Gamma$, it then follows that $\Gamma^\prime$ is also a Fuchsian group. Notice further, by the Schreier index formula \cite{serretrees}, the rank of $\Gamma^\prime$ is $1 + \lvert G \rvert(r-1)$. \\
$(ii)$ Since $\Gamma^\prime, \Gamma$ are Fuchsian groups, we then have the canonical map between Riemann surfaces $\pi \colon \Gamma^\prime \backslash \UH \to \Gamma \backslash \UH$. Consequently, given that $\varphi$ is surjective, we have $G \simeq \Gamma / \Gamma^\prime$, and from \Cref{normalIntersection} we can conclude $G \simeq \mon$.
\end{proof}

\begin{remark}
Given the group presentation for $G$, then $\kernel$ can be seen as smallest normal subgroup containing the relations in $G$.
\end{remark}

Recall that the free group $\Gamma(2)$ can be identified with the triangle group $\Delta(\infty, \infty, \infty)$. Thus, in \Cref{mainResultFree}, we can view $\Gamma, \Gamma^\prime$ as subgroups of $\Delta(\infty, \infty, \infty)$. We now give another proof of Greenberg's result, this time using subgroups of the triangle group $\Delta(m,n,k)$ with $m,n,k$ finite.

\subsection{Triangle group construction} \label{triangleSection}

The specific triangle group we will be working with is $\Delta(2, n, n-1)$. As stated in \Cref{mainResultTriangle}, let $$\varphi \colon \Delta(2, n, n-1) \to S_n$$ be the homomorphism given by 

\begin{equation*}
    \gamma_0 \mapsto (1 \; 2), \quad \gamma_1 \mapsto (1 \; 2 \; \cdots \; n), \quad \gamma_{\infty} \mapsto (n \; \cdots \; 3 \; 2).
\end{equation*}

\begin{proof}[Proof of \Cref{mainResultTriangle}] 
$(i)$ Since $S_n$ is generated by the transposition $(1\textbf{ }2)$ and the $n$-cycle $(1\textbf{ }2 \textbf{ }\cdots\text{ }n)$ and these are mapped to by $\gamma_0$ and $\gamma_1$ respectively, $\varphi$ is a surjective map. Now, notice if $n \geq 5$, then $\frac{1}{2} + \frac{1}{n} + \frac{1}{n-1} < 1$. In that case, $\Delta(2, n, n-1)$ corresponds to a tiling of the hyperbolic plane. And so, it follows that $\Delta(2, n, n-1)$ is an infinite group \cite{shurman1997geometry}. For the cases with $n < 5$ we go one by one. If $n =2 $, then $\Delta(2, 2, 1)$ is the cyclic group $C_2 \simeq S_2$. If $n = 3$, then $\Delta(2, 3, 2)$ is the dihedral group $D_3 \simeq S_3$. And finally, if $n = 4$, then $\Delta(2, 4, 3)$ is the octahedral group, which is isomorphic to $S_4$ \cite{clark2019algebraic}. \\ 
$(ii)$ Since $\Gamma = \varphi^{-1}(G)$ and $\Gamma^\prime = \kernel$ are subgroups of $\Delta(2, n, n-1)$, by \Cref{triangleFuchsian} both $\Gamma, \Gamma^\prime$ are Fuchsian groups.  \\ 
$(iii)$ Consider the restriction $$\varphi \Big\vert_{\Gamma} \colon \Gamma \to G.$$ 
By the first isomorphism theorem,  $G \simeq \Gamma / \operatorname{ker}\left(\varphi \Big\vert_{\Gamma}\right)$, and since $\Gamma^\prime = \ker\varphi = \operatorname{ker}\left(\varphi \Big\vert_{\Gamma}\right)$, we have $G \simeq \Gamma / \Gamma^\prime$. Consequently, given that both $\Gamma, \Gamma^\prime$ are Fuchsian groups, by \Cref{normalIntersection}, it then follows that 

\begin{equation*}
    G \simeq \Gamma / \Gamma^\prime \simeq \mon,
\end{equation*}

where $\pi \colon \Gamma^\prime \backslash \UH \to \Gamma \backslash \UH$.
\end{proof}

\begin{remark}
One may observe that the homomorphism $\varphi$ in \Cref{triangleSection} is really a projection map $F_2/N \to F_2/K$, where $N$ and $K$ are the normal closures of the relations defining $\Delta(2, n, n-1)$ and $S_n$, respectively. Hence, the kernel of this map can be explicitly computed as $K/N$. 

\begin{comment}
Furthermore, the inverse image $\varphi^{-1}(G)$ can be obtained by computing a single pre-image for each of the generators, and then considering the subgroup $X$ generated by those pre-images. Notice that $X \kernel = \varphi^{-1}(G)$.
\end{comment}
\end{remark}

\begin{section}{Acknowledgements}
The authors would like to thank Dr. Edray Goins, Dr. Duane Cooper and Dr. Alexander Barrios for providing support and guidance during our research project. We would also like to thank the National Science Foundation Grant No. DMS-1659138, the Alfred P. Sloan Foundation, Grant No. G-2017-9876 and the  Mathematical Sciences Research Institute (MSRI) for supporting our research.
 
\end{section}

\bibliographystyle{plain}
\bibliography{references.bib}

\begin{thebibliography}{10}

\bibitem{beardon2012geometry}
Alan~F. Beardon.
\newblock {\em The geometry of discrete groups}, volume~91.
\newblock Springer Science \& Business Media, 2012.

\bibitem{birch1994theory}
Bryan Birch.
\newblock Noncongruence subgroups, covers and drawings.
\newblock {\em The Grothendieck Theory of Dessins d'Enfants}, 200:25, 1994.

\bibitem{clark2019algebraic}
Pete Clark and John Voight.
\newblock Algebraic curves uniformized by congruence subgroups of triangle
  groups.
\newblock {\em Transactions of the American Mathematical Society},
  371(1):33--82, 2019.

\bibitem{drutu2018geometric}
Cornelia Druţu and Michael Kapovich.
\newblock {\em Geometric Group Theory}.
\newblock American Mathematical Society, 2018.

\bibitem{dummit2004abstract}
David~S. Dummit and Richard~M. Foote.
\newblock {\em Abstract algebra}.
\newblock John Wiley \& Sons, Inc., Hoboken, NJ, third edition, 2004.

\bibitem{garcia2001topics}
Emilio~Bujalance Garcia, John William~Scott Cassels, Nigel~James Hitchin,
  et~al.
\newblock {\em Topics on Riemann surfaces and Fuchsian groups}, volume 287.
\newblock Cambridge University Press, 2001.

\bibitem{girondo2012introduction}
Ernesto Girondo and Gabino Gonz{\'a}lez-Diez.
\newblock {\em Introduction to compact Riemann surfaces and dessins d'enfants},
  volume~79.
\newblock Cambridge University Press, 2012.

\bibitem{greenberg1963maximal}
Leon Greenberg.
\newblock Maximal fuchsian groups.
\newblock {\em Bulletin of the American Mathematical Society}, 69(4):569--573,
  1963.

\bibitem{jones1987complex}
Gareth~A Jones and David Singerman.
\newblock {\em Complex functions: an algebraic and geometric viewpoint}.
\newblock Cambridge university press, 1987.

\bibitem{katok1992fuchsian}
Svetlana Katok.
\newblock {\em Fuchsian groups}.
\newblock University of Chicago press, 1992.

\bibitem{serretrees}
Jean-Pierre Serre.
\newblock Trees, 1980.

\bibitem{shurman1997geometry}
Jerry Shurman.
\newblock {\em Geometry of the Quintic}.
\newblock John Wiley \& Sons, 1997.

\bibitem{wolfart2005abc}
J{\"u}rgen Wolfart.
\newblock {\em ABC for Polynomials, Dessins d'Enfants, and Uniformization: a
  Survey}.
\newblock Universit{\"a}tsbibliothek Johann Christian Senckenberg, 2005.

\end{thebibliography}

\end{document}